\documentclass[10pt]{article}
\usepackage[english]{babel}
\usepackage{latexsym,amssymb,amsmath,amsthm,amsfonts}
\usepackage{lineno}
\usepackage{float,graphicx}

\usepackage{url}
\usepackage[explicit]{titlesec}
\usepackage[normalem]{ulem}
\usepackage{lipsum}
\usepackage{subcaption}
\usepackage{graphicx}

\newtheorem{theorem}{\uppercase{Theorem}}[section]
\newtheorem{corollary}{\uppercase{Corollary}}[section]
\newtheorem{conjecture}{\uppercase{Conjecture}}[section]

\newtheorem{observation}[theorem]{Observation}
\theoremstyle{definition}
\newtheorem{definition}{\uppercase{Definition}}[section]

\newtheorem{proposition}{\uppercase{Proposition}}[section]
\newtheorem{problem}{\uppercase{Problem}}[section]

\titleformat{\section}
  {\normalfont\Large\bfseries\filcenter}{}{0em}{{\thesection.\hspace*{ .5 em}{#1}}}
\titleformat{name=\section,numberless}
  {\normalfont\Large\bfseries\filcenter}{}{0em}{{{#1}}}

\titleformat{\subsection}
  {\normalfont\large\bfseries}{}{0em}{{\thesubsection\hspace*{ 1 em}{#1}}}
\titleformat{name=\subsection,numberless}
  {\normalfont\Large\bfseries\filcenter}{}{0em}{{{#1}}}

  \newcommand\blfootnote[1]{%
  \begingroup
  \renewcommand\thefootnote{}\footnote{#1}%
  \addtocounter{footnote}{-1}%
  \endgroup
}

\title{\uppercase{Coloring Decompositions of Complete Geometric Graphs}}
\author{
Clemens Huemer$^{1}$
\and Dolores Lara$^{2}$
\and Christian Rubio-Montiel$^{3,4,5}$
}

\date{
$^1$Departament de Matem{\` a}tiques, Universitat Polit{\`e}cnica de Catalunya, Spain. e-mail: clemens.huemer@upc.edu \\
$^2$Departamento de Computaci{\' o}n, Centro de Investigaci{\' o}n y de Estudios Avanzados del Instituto Polit{\' e}cnico Nacional, Mexico. e-mail: dlara@cs.cinvestav.mx\\
$^3$ Divisi{\' o}n de Matem{\' a}ticas e Ingenier{\' i}a, FES Acatl{\' a}n, Universidad Nacional Aut{\' o}noma de M{\' e}xico. e-mail: christian.rubio@apolo.acatlan.unam.mx\\
$^4$UMI LAFMIA 3175 CNRS at CINVESTAV-IPN, Mexico.\\
$^5$Department of Algebra, Comenius University, Slovakia.\\ [2ex]%
\today{}
}

\begin{document}
\maketitle

\renewcommand\subsection{\@startsection{subsection}{3}{\z@}%
                                     {-3.25ex\@plus -1ex \@minus -.2ex}%
                                     {-1.5ex \@plus -.2ex}%
                                     {\normalfont\normalsize\bfseries}}
\makeatother

\blfootnote
{
\begin{minipage}[l]{0.3\textwidth} \includegraphics[trim=10cm 6cm 10cm 5cm,clip,scale=0.15]{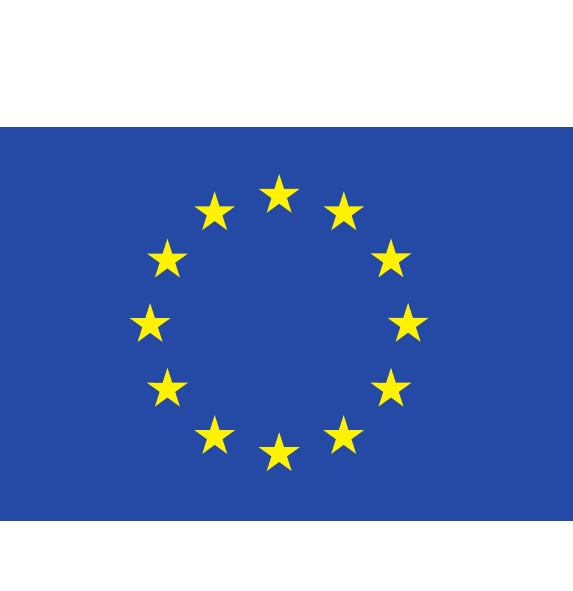} 
\end{minipage}
 \hspace{-2cm} 
 \begin{minipage}[l][1cm]{0.7\textwidth}
 	  This project has received funding from the European Union's Horizon 2020 research and innovation programme under the Marie Sk\l{}odowska-Curie grant agreement No 734922.
\end{minipage}}

\makeatletter

\begin{abstract}
A decomposition of a non-empty simple graph $G$ is a pair $[G,P]$, such that $P$ is a set of non-empty induced subgraphs of $G$, and every edge of $G$ belongs to exactly one subgraph in $P$. The chromatic index $\chi'([G,P])$ of a decomposition $[G,P]$ is the smallest number $k$ for which there exists a $k$-coloring of  the elements  of $P$ in such a way that: for every element of $P$ all of its edges have the same color, and if two members of $P$ share at least one vertex, then they have different colors.
A long standing conjecture of Erd\H{o}s-Faber-Lov\'asz states that every decomposition $[K_n,P]$ of the complete graph $K_n$ satisfies $\chi'([K_n,P])\leq n$. In this paper we work with geometric graphs, and inspired by this formulation of the conjecture, we introduce the concept of chromatic index of a decomposition of the complete geometric graph.  We present bounds for the chromatic index of several types of decompositions when the vertices of the graph are in general position. We also consider the particular case when the vertices are in convex position and present bounds for the chromatic index of a few types of decompositions.
\end{abstract}

\textbf{Keywords:} geometric graphs, coloring, geometric chromatic index.


\section{Introduction}

In 1972, Paul Erd{\H o}s, Vance Faber and L{\' a}szl{\' o} Lov{\'a}sz \cite{MR0409246} made the following conjecture:

\begin{conjecture}
Let $|A_i|=n$, $1\leq i \leq n$, and suppose that $|A_i\cap A_j|\leq1$, for $1\leq i < j \leq n$, then one can color the elements of the union $\bigcup_{i=1}^{n}A_i$ by $n$ colors, so that every set contains elements of all colors.
\end{conjecture}

This conjecture is called the Erd{\H o}s--Faber--Lov{\'a}sz conjecture (for short EFL-conjecture) which has been formulated in terms of decompositions of graphs \cite{2015arXiv150805532A, APVA2016}.

A \emph{decomposition} of a non-empty simple graph $G$ is a pair $[G,P]$, such that (1) $P$ is a set of non-empty induced subgraphs of $G$, and (2) every edge of $G$ belongs to exactly one subgraph in $P$. We can think of $P$ as being a partition of the edges of $G$. 

A $k$-$P$-\emph{coloring} (or a $k$-$P$-proper coloring) of a decomposition $[G,P]$ is a function that assigns to each edge of $G$ a color from a set of $k$ colors, so that (1) for every $H\in P$ all of its edges have the same color, and (2) for any two $H_1,H_2\in P$, if the two graphs have at least one vertex in common, then $E(H_1)$ and $E(H_2)$ have different colors. The \emph{chromatic index} $\chi'([G,P])$ of a decomposition $[G,P]$ is the smallest number $k$ for which there exists a $k$-$P$-coloring of $[G,P]$.

The authors of \cite{2015arXiv150805532A} and \cite{APVA2016}  give an equivalent formulation of the well-known EFL-conjecture in terms of the chromatic index of decompositions of the complete graph:

\begin{conjecture}\label{conjecture1.2}
Every decomposition $[K_n,P]$ of the complete graph $K_n$ satisfies  \[\chi'([K_n,P])\leq n.\]
\end{conjecture}

Note that for every decomposition $[K_n,P]$ of the complete graph $K_n$, since $P$ is a set of non-empty \emph{induced} subgraphs of $K_n$, the set $P$ consists of \emph{complete} subgraphs of $K_n$. To avoid repetition, we do not reiterate this in what follows, but we ask the reader to keep it in mind throughout the paper.

In this paper, inspired by Conjecture \ref{conjecture1.2}, we study a problem about decompositions of complete geometric graphs. A \emph{geometric graph} $\mathsf{G}$\footnote{To avoid any possible confusion, geometric graphs will be denoted by Sans Serif Math letters.} is a drawing in the plane of a graph $G$, such that its vertices are points in general position (no three of them are collinear), and its edges are straight-line segments. Two geometric graphs have \emph{nonempty intersection} if either (1) they have a common vertex, or (2) there exists a pair of edges, one from each graph, that crosses. A \emph{complete geometric graph} of order $n$, denoted $\mathsf{K}_n$, is a graph in wich each vertex is adjacent to every other vertex.

A decomposition $[G,P]$ of a graph $G$ naturally induces a decomposition $[\mathsf{G},P]$ of any geometric graph $\mathsf{G}$. A $k$-$P$-\emph{coloring} (or a $k$-$P$-proper coloring) of a decomposition $[\mathsf{G},P]$ maps to the edges of $\mathsf{G}$ a color from a set with $k$ colors, so that (1) for every $\mathsf{H}\in P$ all of its edges have the same color, and (2) for  any two $\mathsf{H}_1,\mathsf{H}_2\in P$, if the two graphs have nonempty intersection, then $E(\mathsf{H_1})$ and $E(\mathsf{H_2})$ have different colors. The smallest positive integer $k$ for which there is a $k$-$P$-coloring of $[\mathsf{G},P]$ is the \emph{chromatic index} of a decomposition $[\mathsf{G},P]$ and it is denoted by $\chi'([\mathsf{G},P])$. 

In this paper, we give lower and upper bounds for the chromatic index of several families of decompositions (in complete subgraphs) of complete geometric graphs. Furthermore, we state a conjecture on an upper bound on $\chi'([\mathsf{G},P])$ for any partition $P$. The paper is organized as follows. In Section \ref{section2}, we consider complete geometric graphs whose set of vertices are points in general position, and prove the following results. \\

\noindent \textbf{Proposition \ref{proposition1}} Let $\mathsf{K}_n$ be a complete geometric graph of order $n$, 
every decomposition of $\mathsf{K}_n$ has chromatic index at most $\frac{n^2}{6} + O(n^{3/2})$.\\

\noindent \textbf{Theorem \ref{theorem3}} Let $\mathsf{K}_n$ be a complete geometric graph of order $n$, there are decompositions of $\mathsf{K}_n$ with chromatic index at least $\frac{n^2}{24.5}$.\\

Imposing a restriction to the partition, we prove:\\

\noindent\textbf{Proposition \ref{proposition2}} Let $\mathsf{K}_n$ be a complete geometric graph of order $n$, every decomposition of $\mathsf{K}_n$, that does not contain any triangle, has chromatic index at most $\frac{n^2}{12} + O(n^{3/2})$.\\

\noindent\textbf{Theorem \ref{theorem5}} Let $\mathsf{K}_n$ be a complete geometric graph of order $n$, there is a decomposition of $\mathsf{K}_n$ for which most of its elements are triangles, and with chromatic index at most $\frac{n^2}{9}+O(n^{3/2})$.\\

In Section \ref{section3}, we consider the case in which the vertices of the complete geometric graph are in convex position (they are the vertices of a convex polygon). We denote such a graph as $\mathsf{K}^\mathrm{c}_n$. We prove that:\\

\noindent\textbf{Theorem \ref{theorem4}} There are decompositions of $\mathsf{K}^\mathrm{c}_n$ with chromatic index at least $\frac{n^2}{9} - O(n)$.\\

\noindent\textbf{Theorem \ref{theorem3.2}} There are decompositions of $\mathsf{K}^\mathrm{c}_n$ into triangles with chromatic index at most $\frac{n^2}{36}+\Theta(n)$.\\

\noindent\textbf{Theorem \ref{theorem3.3}} Every decomposition of $\mathsf{K}^\mathrm{c}_n$ into triangles has chromatic index at least $\frac{n^2}{119}-O(n)$.\\

Finally, in Section \ref{section4}, we state the following conjecture for complete geometric graphs:\\

\noindent\textbf{Conjecture \ref{conjecture3}} Let $\mathsf{K}_n$ be any complete geometric graph of order $n$. The chromatic index of each decomposition of $\mathsf{K}_n$ is at most $\frac{n^2}{9}+\Theta(n)$.

We also deduce similar conjectures for other variants.


\subsection{Notation, Terminology and Definitions}\label{notation}

Throughout this paper we assume that all sets of points in the plane are in general position. Note that every set of $n$ points in the plane induces a complete geometric graph. Let $S$ be a set of $n$ points in the plane and let $\mathsf{K}_n$ be the complete geometric graph induced by $S$. For brevity, we refer to the points in $S$ as vertices of $\mathsf{K}_n$, and to the straight-line segments connecting two points in $S$ as the edges of $\mathsf{K}_n$.

In the remainder of this subsection we present some definitions related with designs. 

\begin{definition}
Given integers $n$ and $\kappa$, a $2$-$(n,\kappa)$-\emph{design} $\mathcal{D}$ 
is a pair $(\mathcal{P},\mathcal{B})$, where
 $\mathcal{P}$ is a set of $n$ elements, and $\mathcal{B}$ is a collection  of $\kappa$-subsets  of $\mathcal{P}$  with the property that each subset of $\mathcal{P}$ of size $2$ is a subset of exactly one member of $\mathcal{B}$. The members of  $\mathcal{P}$ are called points and the members of $\mathcal{B}$ are called blocks.
\end{definition}

  \begin{definition}
  $\mathcal{D} = (\mathbb{Z}_n,\mathcal{B})$ is a \emph{cyclic design} if $\mathcal{B}$ is fixed by the automorphism $i\mapsto i+1$.
  \end{definition}

\noindent When it is clear from the context, we refer to a $2$-$(n,\kappa)$-design simply as \emph{$2$-design}. $2$-designs are commonly known in the literature as BIBD, standing for balanced incomplete block design.

We can think about any $2$-design $\mathcal{D}$ as a decomposition of the complete graph. That is, $\mathcal{D}$ is a particular case of a decomposition of the complete graph, in which, in the natural way, every block is an element of the partition, see \cite{MR3281127}.

It was proven in \cite{MR679425} that any $2$-design has chromatic index at most $\frac{\kappa n}{\kappa-1}$, and that, every cyclic $2$-design satisfies the EFL conjecture. However, in general, the EFL conjecture is open for $2$-designs, even for $\kappa=3$.

A $2$-$(n,3)$-design is called \emph{Steiner triple system}, and it is denoted by $STS(n)$. It is well-known that $STS(n)$ exists if and only if $n$ is congruent to $1$ or $3$ modulo $6$ (see \cite{MR0314644}). It is also well-known that a cyclic $STS(n)$ exists if and only if $n \not= 9$ and $n$ is congruent to $1$ or $3$ modulo $6$ (see \cite{MR1843379,MR1557027}). Figure \ref{Fig1} shows the unique $STS(9)$, see \cite{MR2661401}.

\begin{figure}
\begin{center}
\includegraphics{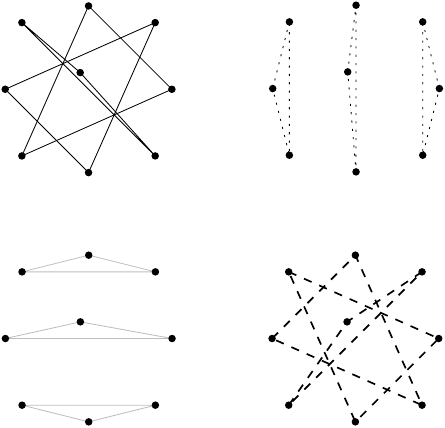}
\caption{The unique Steiner triple system $STS(9)$ as a decomposition of $K_9$.}\label{Fig1}
\end{center}
\end{figure}

There is another special type of well-known $2$-design: A \emph{projective plane} is a $2$-$(q^2+q+1,q+1)$-design having the following properties:
\begin{enumerate}
\item Given any two distinct blocks there is exactly one point incident with both of them.
\item There are four points so that no line is incident with more than two of them.
\end{enumerate}

The number $q$ is called the \emph{order} of the projective plane. A projective plane of order $q$ is denoted by $\Pi_q$. It is known \cite{MR1612570} that there exists at least one projective plane of order $q$ for any prime power $q$. 

\section{Points in General Position}\label{section2}

Let $\mathcal{D}$ be a $2$-$(n,2)$-design. $\mathcal{D}$ induces the decomposition $[K_n,E(K_n)]$ of $K_n$. Therefore, Vizing's theorem verifies the EFL conjecture, since $\chi'([K_n,E(K_n)])$ is the usual chromatic index $\chi'(K_n)$ of the complete graph, and $\chi'(K_n) \leq n$.

In the case of the complete geometric graph $\mathsf{K}_n$, 
$\chi'([\mathsf{K}_n,E(\mathsf{K}_n)])$ has also been studied before as the chromatic index $\chi'(\mathsf{K}_n)$ of $\mathsf{K}_n$. In \cite{MR2155418} the authors prove that $n \leq \chi'(\mathsf{K}_n) \leq c n^{3/2}$, for some constant $c > 0$ and any complete geometric graph $\mathsf{K}_n$. See \cite{MR3461960} for similar results.

Now, we can state an upper bound for the chromatic index of
every decomposition of the complete geometric graph.

\begin{proposition}\label{proposition1}
Let $[\mathsf{K}_n,P]$ be a decomposition of $\mathsf{K}_n$, then \[\chi'([\mathsf{K}_n,P])\leq \frac{n^{2}}{6}+O(n^{3/2}).\]
\end{proposition}
\begin{proof}
We know that every element of $P$, that is not an edge, contains at least three edges. Since at most $c n^{3/2}$ colors are needed to color the elements of $P$ that are edges,
 $\chi'([\mathsf{K}_n,P])$ is at most $\binom{n}{2}/3+cn^{3/2},$ and the result follows.
\end{proof}

We can give a better upper bound when $P$ has no triangles. Note that, in this case, every element of $P$ that is not an edge, contains at least six edges.

\begin{proposition}\label{proposition2}
Let $[\mathsf{K}_n,P]$ be a decomposition of $\mathsf{K}_n$ such that $P$ does not contain triangles,
then \[\chi'([\mathsf{K}_n,P])\leq \frac{n^{2}}{12}+O(n^{3/2}).\]
\end{proposition}

In the following paragraphs, we obtain a lower bound for the chromatic index of a decomposition consisting of any complete geometric graph and a given partition. 


In order to prove our result, first we divide the plane into nine regions and then we use this partition to obtain a decomposition of the edges of the complete graph. To obtain the partition we need we prove a corollary that follows from the following theorem \cite{Buck06}, originally proved by Ceder \cite{MR0188890}.

\begin{theorem}[\cite{Buck06, MR0188890}]\label{thm:buck}
Let $\mu$ be a finite measure absolutely continuos with respect to the Lebesgue measure on $\mathbb{R}^2$. Then there are three concurrent lines that partition the plane into six parts of equal measure. 
\end{theorem}

From this theorem we can obtain the following corollary (which will be used in the proof of Theorem~\ref{theorem5}), see also \cite{MR3164167}. 

\begin{corollary}\label{coro}
Let $S$ be a set of $n$ points in general position in the plane. There exist three lines, two of them parallel, that divide $S$ into six parts with at least $\frac{n}{6}-1$ points each.
\begin{proof}

Let $u$ be a unit vector with direction $\theta$. Consider the set of lines with direction $\theta$, denoted $\{\ell^i_\theta \colon  i \in \Delta\}$, where $\Delta$ is a set of indices. Let $\Pi_{\ell_\theta^i}^+$ be the closed positive half-plane defined by $\ell^i_\theta$, and let ${\Pi_{\ell^i_\theta}^-}$ be the closed negative half-plane defined by $\ell^i_\theta$. Let $X$ be a set of $n$ points in general position in the plane. Consider the following set:

$$
\left \{
{\Pi_{\ell_\theta^j}^+} \colon \vert{\Pi_{\ell_\theta^j}^+}\cap X\vert \geq k
\right \} \subseteq \left \{ {\Pi_{\ell_\theta^i}^+} \colon i \in \Delta 
\right \}.$$

That is, the set of all closed positive half-planes defined by $\ell^j$,  $j \in \Delta$, and such that the cardinality of their intersection with $X$ is at least $k$. Analogously, for the negative half-planes, consider the set: 

$$
\left \{
{\Pi_{\ell_\theta^j}^-}\colon \vert{\Pi_{\ell_\theta^j}^-}\cap X\vert \geq k
\right \} \subseteq \left \{ {\Pi_{\ell_\theta^i}^-} \colon i \in \Delta 
\right \}.
$$

Let $(a,b)$ be a pair of real numbers such that $0 < a,b < 1$ and $a+b=1$. Let

$$\ell_\theta^+(X; a) = \partial \bigcap_j \left \{
{\Pi_{\ell_\theta^j}^+}\colon \vert{\Pi_{\ell_\theta^j}^+}\cap X\vert \geq \lceil an \rceil
\right \} \text{ and}$$ 

$$\ell_\theta^-(X; b) = \partial \bigcap_j 
\left \{
{\Pi_{\ell_\theta^j}^-}\colon \vert{\Pi_{\ell_\theta^j}^-}\cap X\vert \geq \lceil bn \rceil
\right \}.$$

Denote by $\ell_\theta(X; a,b) := \ell_\theta^+(X;a) \cdot \ell_\theta^-(X;b)$ the line at the same distance from each of the lines $l^+$ and $l^-$.

Let $\ell_\theta^1 = \ell_\theta\left(S; \frac{1}{3}, \frac{2}{3} \right)$ and  $\ell_\theta^2 = \ell_\theta\left(S; \frac{2}{3}, \frac{1}{3} \right)$. Also, let $A = S \cap {\Pi_{\ell_\theta^1}^+}$ and $B = S \cap {\Pi_{\ell_\theta^2}^-}$. That is, $A$ is the point set lying in the positive half-plane defined by $\ell_\theta^1$, and $B$ is the point set lying in the negative half-plane defined by $\ell_\theta^2$. Both $A$ and $B$ contain $\lceil\frac{n}{3} \rceil$  points. Let
$\ell_\theta^A = \ell_\theta\left(A; \frac{1}{2},\frac{1}{2}\right)$ and 
$\ell_\theta^B = \ell_\theta\left(B; \frac{1}{2},\frac{1}{2}\right)$.

Note that, independently of $\theta$, the slope of the lines $\ell_\theta^A$ and $\ell_\theta^B$ can be changed continuously until it becomes equal to $\theta + \pi$. Thus there exist a unique slope $\theta^*$, for which the two lines are the same; that is, $\ell_{\theta^*}^A = \ell_{\theta^*}^B$. We denote this line by $\ell^3_{\theta}$.

Let $C = (S \setminus A \cup B) \cap {\Pi^+_{\ell_\theta^3}}$ and $D = (S \setminus A \cup B) \cap {\Pi^-_{\ell_\theta^3}}$. If $|C| \geq \frac{n}{6} - 1$ and $|D| \geq \frac{n}{6}-1$, then the proof is completed. Let us assume, without loss of generality, that $|C| > |D|$ and $|D| < \frac{n}{6}-1$. The choice of the lines  $\ell_\theta^1, \ell_\theta^2$ and $\ell_{\theta}^3$ depend on $\theta$. For each $\theta$ we have $\ell_\theta^1 = \ell_{-\theta}^2$, $\ell_\theta^2 = \ell_{-\theta}^1$ and $\ell_{\theta}^3 = \ell_{-\theta}^3$. Also, since for $\ell_{\theta}^3$  we have $|C| > |D|$ and $|D| < \frac{n}{6}-1$, then for $\ell_{-\theta}^3$ we have  $|D| > |C|$ and $|C| < \frac{n}{6}-1$. For continuity it follows that there exist a direction for which $|C| \geq \frac{n}{6} - 1$ and $|D| \geq \frac{n}{6}-1$. This completes the proof.
\end{proof}
\end{corollary}

Now we describe the partition of the plane. Let $S$ be a set of $n=7q+6$ points in general position in the plane, where $q$ is a prime power greater than 2. Let $l_1$, $l_2$, and $l_3$ be three horizontal lines, listed from top to bottom. Let $S' \subseteq S$ be the set  of points between $l_1$ and $l_2$, and let $S_7 \subseteq S$ be the set of points between $l_2$ and $l_3$. It is clear that we can choose the three lines so that $|S'|=6q+6$ and $|S_1|=q$. Furthermore, by Theorem~\ref{thm:buck} there are three concurrent lines that divide the set $S'$ into $6$ parts each containing $q$ points of $S'$ in its interior. Name this lines  $l_4$, $l_5$, and $l_6$, respectively, and label the six sets as $S_2,S_3,S_4,S_5,S_6,S_7$, listed in clockwise order around $p$, where $p$ is the point of intersection of the three lines. Refer to Figure \ref{fig:Fig0_a}. 

\begin{figure}

\begin{subfigure}{0.4\textwidth}
\includegraphics[width=\linewidth]{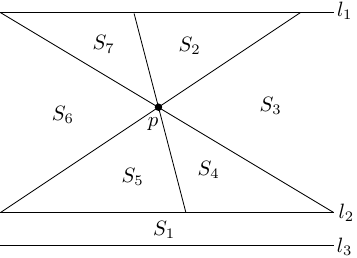}
\caption{  } \label{fig:Fig0_a}
\end{subfigure}
\hspace*{\fill} 
\begin{subfigure}{0.4\textwidth}
\includegraphics[width=\linewidth]{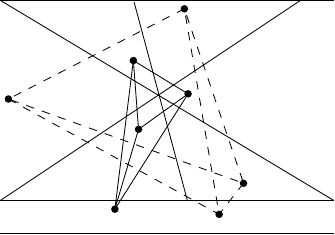}
\caption{ } \label{fig:Fig0_b}
\end{subfigure}
\caption{(a) The line configuration. (b) One subgraph $\mathsf{X}_{i,j}$ represented with dashed edges, and one subgraph $\mathsf{Y}_{i',j'}$ represented with solid edges.} \label{fig:1}
\end{figure}

Since $q$ is a prime power, there exists at least one projective plane of order $q$. Let $\Pi_q$ be a projective plane of order $q$, and let $z$ and $ \mathfrak{p}_{i,j}, 1 \leq i \leq 4, 1 \leq j \leq q$, be points in $\Pi_q$. Let four lines incident with $z$ be $\ell_1=\{\mathfrak{p}_{1,1}, \ldots, \mathfrak{p}_{1,q},z\}$, $\ell_2=\{\mathfrak{p}_{2,1}, \ldots, \mathfrak{p}_{2,q},z\}$, $\ell_3=\{\mathfrak{p}_{3,1}, \ldots, \mathfrak{p}_{3,q},z\}$ and $\ell_4=\{\mathfrak{p}_{4,1}, \ldots, \mathfrak{p}_{4,q},z\}$. (This type of design is known in the literature as transversal design, in our case we are using a $TD(q+1,q)$-design.) Use these four lines, and the correspondence $(a,b) \mapsto v_{a,b} $ and $(a,b) \mapsto u_{a,b} $, to label the points of the subsets $S_i$, as follows. Using $\ell_1\setminus z$, label the points of $S_2$ as $\{v_{1,1}, \ldots, v_{1,q}\}$, and the points of $S_3$ as $\{u_{1,1}, \ldots, u_{1,q}\}$. Similarly, using $\ell_2\setminus z$ label the points of $S_4$ as $\{v_{2,1}, \ldots, v_{2,q}\}$, and the points of $S_5$ as $\{u_{2,1}, \ldots, u_{2,q}\}$. Using $\ell_3\setminus z$ label the points of $S_6$ as $\{v_{3,1}, \ldots, v_{3,q}\}$, and the points of $S_7$ as $\{u_{3,1}, \ldots, u_{3,q}\}$. Finally, using $\ell_4\setminus z$ label the points of $S_1$ as $\{v_{4,1}, \ldots, v_{4,q}\}$.

Now, using points of $S$, we construct some complete geometric graphs.

For each $i$ and $j$ in $\{1, \ldots, q\}$, let $\mathsf{X}_{i,j}$ be a complete geometric graph of order four. 
Next we choose the four vertices of $\mathsf{X}_{i,j}$. 
Consider the line $\overline{(\mathfrak{p}_1,\mathfrak{p}_i)(\mathfrak{p}_4,\mathfrak{p}_j)}\in \Pi_q$ induced by the points $(\mathfrak{p}_1,\mathfrak{p}_i)$ and $(\mathfrak{p}_4,\mathfrak{p}_j)$. Note that $(\mathfrak{p}_1,\mathfrak{p}_i)$ belongs to $\ell_1$, and that $(\mathfrak{p}_4,\mathfrak{p}_j)$ belongs to $\ell_4$. In accordance with the labeling in the paragraph above, the point $(\mathfrak{p}_1,\mathfrak{p}_i)$ corresponds to the point $v_{1,i} \in S_2$, and the point $(\mathfrak{p}_4,\mathfrak{p}_j)$ corresponds to the point $v_{4,j} \in S_1$. These two points are two vertices of $\mathsf{X}_{i,j}$. Now consider $i'$ and $j'$ so that $\overline{(\mathfrak{p}_1,\mathfrak{p}_i)(\mathfrak{p}_4,\mathfrak{p}_j)}=\overline{(\mathfrak{p}_2,\mathfrak{p}_i')(\mathfrak{p}_3,\mathfrak{p}_j')}$. As before, the point $(\mathfrak{p}_2,\mathfrak{p}_i')$ corresponds to the point $v_{2,i'} \in S_4$, and the point $(\mathfrak{p}_3,\mathfrak{p}_j')$ corresponds to the point $v_{3,j'} \in S_6$. These two points are the other two vertices of $\mathsf{X}_{i,j}$. That is, the set of vertices of $\mathsf{X}_{i,j}$ is $\{v_{1,i}, v_{4,j}, v_{2,i'}, v_{3,j'}\}$. Refer to Figure \ref{fig:Fig0_b}. Note that there are  exactly $q^2$ of these graphs.

For each $i$ and $j$ in $\{1, \ldots, q\}$, let $\mathsf{Y}_{i,j}$ be a complete geometric graph of order four. 
Next we choose the four vertices of $\mathsf{Y}_{i,j}$. Consider the line $\overline{(\mathfrak{p}_1,\mathfrak{p}_i)(\mathfrak{p}_4,\mathfrak{p}_j)}\in \Pi_q$ induced by the points $(\mathfrak{p}_1,\mathfrak{p}_i)$ and $(\mathfrak{p}_4,\mathfrak{p}_j)$. Note that $(1\mathfrak{p}_,\mathfrak{p}_i)$ belongs to $\ell_1$, and that $(\mathfrak{p}_4,\mathfrak{p}_j)$ belongs to $\ell_4$. In accordance with the labeling in the paragraph above, the point $(\mathfrak{p}_1,\mathfrak{p}_i)$ corresponds to the point $u_{1,i} \in S_3$, and the point $(\mathfrak{p}_4,\mathfrak{p}_j)$ corresponds to the point $v_{4,j} \in S_1$. These two points are two vertices of $\mathsf{Y}_{i,j}$. Now consider $i'$ and $j'$ so that $\overline{(\mathfrak{p}_1,\mathfrak{p}_i)(\mathfrak{p}_4,\mathfrak{p}_j)}=\overline{(\mathfrak{p}_2,\mathfrak{p}_i')(\mathfrak{p}_3,\mathfrak{p}_j')}$. As before, the point $(\mathfrak{p}_2,\mathfrak{p}_i')$ corresponds to the point $u_{\mathfrak{p}_2,\mathfrak{p}_i'} \in S_5$, and the point $(\mathfrak{p}_3,\mathfrak{p}_j')$ corresponds to the point $u_{\mathfrak{p}_3,\mathfrak{p}_j'} \in S_7$. These two points are the other two vertices of $\mathsf{Y}_{i,j}$. That is, the set of vertices of $\mathsf{Y}_{i,j}$ is $\{u_{1,i}, v_{4,j}, u_{2,i'}, u_{3,j'}\}$. Refer to Figure \ref{fig:Fig0_b}. Note that there are exactly $q^2$ of these graphs.

The following two observations were proven in \cite{MR3461960}. (We omit both proofs.)

\begin{observation}\label{lemma1}
The point $p$ is inside each of the triangles induced by the graphs $\mathsf{X}_{i,j}-v_{4,j}$ and $\mathsf{Y}_{i,j}-v_{4,j}$, defined above.
\end{observation}

\begin{observation}\label{lemma2}
Every two graphs $\mathsf{X}_{i,j}$ and $\mathsf{Y}_{i,j}$, intersect. That is, there is a pair of edges, one from each graph, which cross.
\end{observation}

By construction, every pair of points, one from $S_7$ and one from $S_1$, define a complete geometric graph $\mathsf{X}_{i,j}$. Similarly, every pair of points, one from $S_3$ and one from $S_1$, define a complete geometric graph $\mathsf{Y}_{i,j}$. These graphs are edge-disjoint and pairwise intersecting. From these observations, we can get the following theorem.

\begin{theorem}\label{theorem3}
For every natural number $n$, there exists a decomposition $[\mathsf{K}_n,P]$ of $\mathsf{K}_n$ such that \[\chi'([\mathsf{K}_n,P])\geq \frac{n^{2}}{24.5}-\Theta(n).\]
\end{theorem}

\begin{proof} Observation~\ref{lemma2} and the prime number theorem imply that for every positive integer $n$, any $\mathsf{K}_n$ has a set of at least $2(\frac{n}{7})^2-\Theta(n)$ edge disjoint complete geometric subgraphs $\mathsf{G}_{i,j}$ ($\mathsf{G}_{i,j}\in\{\mathsf{X}_{i,j},\mathsf{Y}_{i,j}\}$) which are pairwise intersecting. Any partition $P$ containing the graphs $\mathsf{G}_{i,j}$ must assign a different color to each of these graphs in any $\chi([\mathsf{K}_n,P])$-$P$-coloring of $[\mathsf{K}_n,P]$.
\end{proof}


Triangles play an important role in decompositions. In the case of Steiner Triple Systems $STS(n)$ it is known  that the EFL-conjecture is true for $n\leq 19$ \cite{MR2661401}. In the geometric setting triangles seem to be important too: the bound given in Theorem~\ref{theorem3} can be seen as consisting mostly of triangles, however, if we restrict the decomposition to not have any triangles, Proposition~\ref{proposition2} shows that the chromatic index decreases significantly.  Notice, however, that the set of triangles induced by the proof of Theorem~\ref{theorem3} contains exactly the same point in common, this property is stronger than the one required in our definition of intersection: two triangles intersect if they share a common interior point or a vertex. A natural question is:  if we loose this strong restriction, how does the chromatic index of decompositions consisting of triangles behave? 

To end this section,  we show that there is a decomposition of the complete geometric graph that consists mostly of triangles, and with chromatic index at most $\frac{n^2}{9}$.

\begin{theorem}\label{theorem5}
For any sufficiently large natural number $n$ there exists a decomposition $[\mathsf{K}_n,P]$ of $\mathsf{K}_n$ such that (1) every element of $P$ is a triangle, except for $o(n^2)$ edges, and (2) \[\chi'([\mathsf{K}_n,P])\leq \frac{n^{2}}{9}+O(n^{3/2}).\]
\end{theorem}
\begin{proof}

First, we divide the plane into nine regions, and then we use these regions to construct a partition of the edges of $\mathsf{K}_9$.

We divide the plane using a specific configuration of lines. Next, we describe such configuration. Let $S$ be a set of $n$ points in general position in the plane. Applying an affine transformation, by Corollary~\ref{coro}, there are two vertical lines $l_1$ and $l_2$, and one horizontal line $l_3$, so that they divide the set $S$ into $6$ parts of equal size. That is, each part contains at least $\frac{n}{6}-1$ points in its interior (at most $c<6$ points are not considered). We label each one of the six sets as $S_1,S_2,S_3,S_4,S_5,S_6$, refer to Figure~\ref{fig:Fig3}. 

It is known \cite{MR1851081} that for every positive integer $x$, the interval $[x-o(x^{0.525}),x]$ contains prime power numbers. Let $q$ be the largest of these primer powers for $x=\left\lfloor n/9 \right\rfloor$.

Divide the region containing each set $S_i$ using a line parallel to $l_3$, into two regions $R_i$ and $R_i'$, refer to Figure~\ref{fig:Fig3a}.  Each of these lines is  chosen so that every $R_i$ contains $q$ points. 
Then, define the regions $R_7=R_1'\cup R_4'$, $R_8=R_2'\cup R_5'$ and $R_9=R_3'\cup R_6'$. 

Since $q$ is a prime power, there exists a projective plane of order $q$. Let $\Pi_q$ be a projective plane of order $q$, and let $z$ and $ \mathfrak{p}_{i,j}, 1 \leq i \leq 9, 1 \leq j \leq q$, be points in $\Pi_q$. Let nine lines incident with $z$ be $\ell_i=\{\mathfrak{p}_{i,1}, \ldots, \mathfrak{p}_{i,q},z\}$. Use these nine lines, and the usual correspondance  $(a,b) \mapsto v_{a,b} $, to label the points in  $R_i$ as $\{v_{i,1}, \ldots, v_{i,q}\}$, such that each point in $R_i$ corresponds to a point in $\ell_i\setminus z$.

Next, we construct a decomposition $P$ of $\mathsf{K}_n$ such that every element of $P$ is a triangle. At the same time, we give a proper coloring of the elements of the partition.

Let $T_1=S_1\cup S_4$, $T_2=S_2\cup S_5$, and $T_3=S_3\cup S_6$. In the next paragraph we describe a partial partition of the edges of $\mathsf{K}_n$ that uses all the edges between $T_1$, $T_2$, and $T_3$.

Since every two points in $\Pi_q$ determine exactly one line, choosing a pair of points from two different regions $R_i$ and $R_{i'}$ (for $i,i' \in \{1, \ldots, 9\}$) is sufficient to induce one fixed geometric graph $\mathsf{K}_9$. Note that such graph has exactly one point from each of the nine regions. Let $\{v_{1,j_1},\dots,v_{9,j_9}\}$, where $j_1, \ldots, j_9 \in \{1, \ldots, q\}$, be the vertices of one fixed $\mathsf{K}_9$. As we mentioned in Section \ref{notation}, $\mathsf{K}_9$ has a partition into twelve triangles. The partition consists of four classes, which are pairwise non-crossing. Each class has exactly three triangles. See Figure \ref{Fig1}. Choose one of these classes as the triangles $(v_{1,j_1},v_{4,j_4},v_{7,j_7})$, $(v_{2,j_2},v_{5,j_5},v_{8,j_8})$ and $(v_{3,j_3},v_{6,j_6},v_{9,j_9})$, and one more as $(v_{1,j_1},v_{2,j_2},v_{3,j_3})$, $(v_{4,j_4},v_{5,j_5},v_{6,j_6})$ and $(v_{7,j_7},v_{8,j_8},v_{9,j_9})$. The remaining triangle classes are determined by this choice. The triangles $(v_{1,j_1},v_{2,j_2},v_{3,j_3})$ and $(v_{4,j_4},v_{5,j_5},v_{6,j_6})$ are separated by the line $l_3$. We assign to these two triangles the same color. Furthermore, we assign one different color to each one of the seven triangles having edges between $T_1$, $T_2$ and $T_3$. See Figure \ref{fig:Fig3b}. Note that, in fact we can assign the same color to the triangles $(v_{1,j_1},v_{4,j_4},v_{7,j_7})$, $(v_{2,j_2},v_{5,j_5},v_{8,j_8})$ and $(v_{3,j_3},v_{6,j_6},v_{9,j_9})$.  Therefore, $8$ colors are sufficient to color the edges in the partial partition of $\mathsf{K}_n$, which uses all the edges between $T_1$, $T_2$ and $T_3$.
We repeat the process in every part $T_1$, $T_2$ and $T_3$ using recursion. The total number of colors used in the triangles is $T(n)=8q^2+T(2\left\lfloor n/6 \right\rfloor)\leq 8(n/9)^2 + T(n/3)$ which leads to $T(n)\leq  n^2/9 + O(n).$ The uncolored edges require at most $O(n^{3/2})$ colors. Therefore, the result follows.
\end{proof}

\begin{figure}
\begin{center}
\includegraphics[scale=0.85]{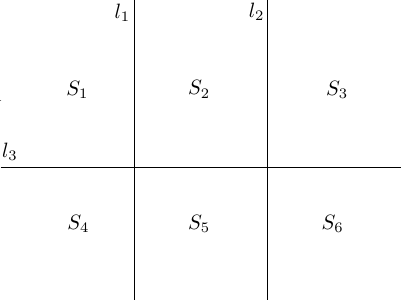}
\caption{The line configuration $\mathcal{M}$.}\label{fig:Fig3}
\end{center}
\end{figure}

\begin{figure}
\begin{center}
\includegraphics[scale=0.85]{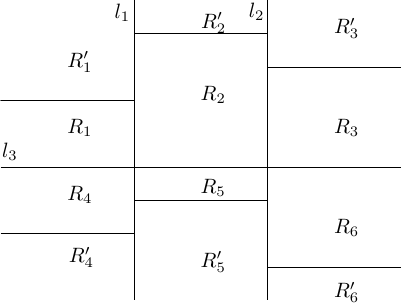}
\caption{The twelve regions $R_i$ and $R'_i$}\label{fig:Fig3a}
\end{center}
\end{figure}

\begin{figure}
\begin{center}
\includegraphics[scale=0.85]{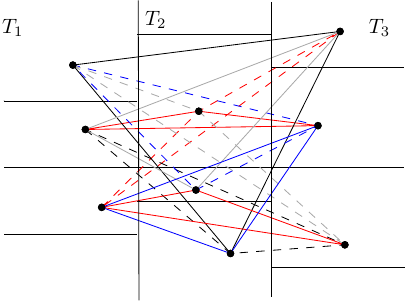}
\caption{Nine triangles, two disjoint triangles are shown in solid red.}\label{fig:Fig3b}
\end{center}
\end{figure}


\section{Points in Convex Position}\label{section3}

In this section, we consider the case in which the vertices of the complete geometric graph are in covex position. We call this type of graph a \emph{complete convex geometric graph}, and we denote it as $\mathsf{K}^\mathrm{c}_n$. The crossing pattern of the edge set of a complete convex geometric graph depends only on the number of vertices, and not
on their particular position. Therefore, without loss of generality we assume that the point set of the graph corresponds to the vertices of a regular polygon.

Let $\mathsf{K}^\mathrm{c}_n$ be a complete convex geometric graph of order $n$, and let $\{1, \ldots, n\}$ be the vertices of the graph listed in clockwise order. In the remainder of this section we exclusively
work with this type of graphs. It is important to bear in mind that all sums are taken modulo $n$, for the sake of simplicity we will avoid writing this explicitly.  We denote by $e_{i,j}$ the edge between the vertices $i$ and $j$.

\emph{Cano et. al.} proved in \cite{MR3205151} that for every set of points in convex position in the plane there are at least $\frac{n^2}{9} - n$ triangles which contain a point in common. It is not hard to see that this set of triangles is also pairwise intersecting; from this result our next theorem follows. However, note that the intersection of the set of such triangles is not empty, and that this is a stronger condition that the one we need. We now prove that if we obtain a decomposition of the complete graph in pairwise-intersecting triangles with not necessarily a point in common, the chromatic index does not decreases.

\begin{theorem}\label{theorem4}
For every natural number $n$ there exists a decomposition $[\mathsf{K}^\mathrm{c}_n,P]$ of $\mathsf{K}^\mathrm{c}_n$ such that \[\chi'([\mathsf{K}^\mathrm{c}_n,P])\geq \frac{n^{2}}{9}-O(n).\]
\end{theorem}
\begin{proof}

Suppose that $3|n$. Let $S_1=\{1,\dots,n/3\}$, $S_2=\{n/3+1,\dots,2n/3\}$ and $S_3=\{2n/3+1,\dots,n\}$. Now consider the geometric complete bipartite graph with vertex set $S_2\cup S_3$ and edge set $M=\{e_{i,j}|i\in S_2, j\in S_3\}$. Then, $M$ can be decomposed into $n/3$ (not necessarily plane) perfect matchings $M_k$, $1\leq k \leq n/3$. Every edge $e_{i,j}$ in $M_k$ defines a triangle with vertex set $\{i,j,k\}$ for any $k\in S_1$. Moreover, these triangles are edge-disjoint and every two triangles have non-empty intersection.
\end{proof}

The following theorem states that there exist decompositions into triangles of $\mathsf{K}^\mathrm{c}_n$ such that its chromatic index is not so high.

\begin{theorem}\label{theorem3.2}
Let $n=18k+1$ with $k$ even. There exists a decomposition $[\mathsf{K}^\mathrm{c}_n,P]$ of $\mathsf{K}^\mathrm{c}_n$ such that each element of $P$ is a triangle and \[\chi'([\mathsf{K}^\mathrm{c}_n,P])\leq \frac{n^{2}}{36}+\Theta(n).\]
\end{theorem}
\begin{proof}
We use the mapping $i\mapsto i+1$ mod $n$. Consider an edge $e_{i,j}$. The orbit of this edge is $\{e_{i+x,j+x}:1\leq x \leq n\}$. Since $n$ is odd, every orbit has $n$ edges. The orbit of $e_{i,j}$ is determined completely by specifying the \emph{minimum difference} or \emph{length} $\min\{i-j,j-i\}$ mod $n$, hence, a number $d_{i,j}$ is a representation of the orbit of $e_{i,j}$ where $1\leq d_{i,j} \leq n/2$.

The orbit of a triangle $\{i,j,k\}$ is defined similarly by the triple $(d_{i,j},d_{j,k},d_{i,k})$. Every orbit has $n$ triangles when $3\nmid n$. The three lengths are not sufficient to determine the orbits of triples, therefore, $(d_{i,j},d_{j,k},d_{i,k})$ is a \emph{difference triple} if each entry is at most $n/2$,  $d_{i,j}<d_{j,k}<d_{i,k}$ and either $d_{i,j}+d_{j,k}+d_{i,k}\equiv 0$ mod $n$ or $d_{i,j}+d_{j,k}\equiv d_{i,k}$ mod $n$, see Chapter 7 of \cite{MR1843379}.

In \cite{MR1557027} the following partition $P$ of $K_n$ arising from a cyclic Steiner Triple System $STS(n)$ is given: Table \ref{Table} shows in the columns $E_i$ the minimum differences ($1\leq d_{i,j} \leq 9k$) and it is divided into $3$ types of orbits $(E_1,E_2,E_3)$, $(E_4,E_5,E_6)$ and $(E_7,E_8,E_9)$, each of them has $k$ orbits. To match Table \ref{Table} with the table given in \cite{MR1557027} pg. 253, consider the column $(E_4,E_5,E_6)$ in the inverse order.

We consider the induced partition of $\mathsf{K}^\mathrm{c}_n$ by $P$. Every element $t$ in the column ``Boxes'' is a set of $n$ colors, namely, $\{1+n(t-1),\dots,n+n(t-1)\}$. The set of triples $(d_{i,j},d_{j,k},d_{i,k})$ in the rows $t$ are colored with the color $n+n(t-1)$ and the corresponding other $n-1$ elements of the orbits, $(d_{i,j}+s,d_{j,k}+s,d_{i,k}+s)$ where $1\leq s\leq n-1$, are colored with the color $s+n(t-1)$ respectively.

\setlength\tabcolsep{3pt}
\begin{table}
\begingroup\makeatletter\def\f@size{9}\check@mathfonts
\begin{tabular}{lll|lll|llll}
$E_{1}$&$E_{2}$&$E_{3}$&$E_{4}$&$E_{5}$&$E_{6}$&$E_{7}$&$E_{8}$&$E_{9}$&$Boxes$\tabularnewline
\hline
\hline
&&&\textbf{$(3k,$}&\textbf{$3k+1,$}&\textbf{$6k+1)$}&\textbf{$(2,$}&\textbf{$8k,$}&\textbf{$8k+2)$}&$\frac{1}{2}k-1$\tabularnewline
&&&$(3k-3,$&$4k+3,$&$7k)$&$(5,$&$8k-1,$&$8k+4)$&$\frac{1}{2}k-2$\tabularnewline
$(1,$&$4k+1,$&$4k+2)$&$(3k-6,$&$4k+5,$&$7k-1)$&$(8,$&$8k-2,$&$8k+6)$&$\frac{1}{2}k-3$\tabularnewline
$(4,$&$4k,$&$4k+4)$&$(3k-9,$&$4k+7,$&$7k-2)$&$(11,$&$8k-3,$&$8k+8)$&$\frac{1}{2}k-4$\tabularnewline
$\vdots$&$\vdots$&$\vdots$&$\vdots$&$\vdots$&$\vdots$&$\vdots$&$\vdots$&$\vdots$&\tabularnewline
$(\tfrac{3}{2}k-17,$&$\tfrac{7}{2}k+7,$&$5k-10)$&$(\tfrac{3}{2}k+12,$&$5k-7,$&$\tfrac{13}{2}k+5)$&$(\tfrac{3}{2}k-10,$&$\tfrac{15}{2}k+4,$&$9k-6)$&$3$\tabularnewline
$(\tfrac{3}{2}k-14,$&$\tfrac{7}{2}k+6,$&$5k-8)$&$(\tfrac{3}{2}k+9,$&$5k-5,$&$\tfrac{13}{2}k+4)$&$(\tfrac{3}{2}k-7,$&$\tfrac{15}{2}k+3,$&$9k-4)$&$2$\tabularnewline
$(\tfrac{3}{2}k-11,$&$\tfrac{7}{2}k+5,$&$5k-6)$&$(\tfrac{3}{2}k+6,$&$5k-3,$&$\tfrac{13}{2}k+3)$&$(\tfrac{3}{2}k-4,$&$\tfrac{15}{2}k+2,$&$9k-2)$&$1$\tabularnewline
$(\tfrac{3}{2}k-8,$&$\tfrac{7}{2}k+4,$&$5k-4)$&$(\tfrac{3}{2}k+3,$&$5k-1,$&$\tfrac{13}{2}k+2)$&$(\tfrac{3}{2}k-1,$&$\tfrac{15}{2}k+1,$&$9k)$&$1$\tabularnewline
\cline{4-10}
$(\tfrac{3}{2}k-5,$&$\tfrac{7}{2}k+3,$&$5k-2)$&$(\tfrac{3}{2}k,$&$5k+1,$&$\tfrac{13}{2}k+1)$&$(\tfrac{3}{2}k+2,$&$\tfrac{15}{2}k,$&$9k-1)$&$2$\tabularnewline
$(\tfrac{3}{2}k-2,$&$\tfrac{7}{2}k+2,$&$5k)$&$(\tfrac{3}{2}k-3,$&$5k+3,$&$\tfrac{13}{2}k)$&$(\tfrac{3}{2}k+5,$&$\tfrac{15}{2}k-1,$&$9k-3)$&$3$\tabularnewline
\cline{1-3}
$\vdots$&$\vdots$&$\vdots$&$\vdots$&$\vdots$&$\vdots$&$\vdots$&$\vdots$&$\vdots$&\tabularnewline
$(3k-23,$&$3k+9,$&$6k-14)$&$(18,$&$6k-11,$&$6k+7)$&$(3k-16,$&$7k+6,$&$8k+11)$&$\frac{1}{2}k-4$\tabularnewline
$(3k-20,$&$3k+8,$&$6k-12)$&$(15,$&$6k-9,$&$6k+6)$&$(3k-13,$&$7k+5,$&$8k+9)$&$\frac{1}{2}k-3$\tabularnewline
$(3k-17,$&$3k+7,$&$6k-10)$&$(12,$&$6k-7,$&$6k+5)$&$(3k-10,$&$7k+4,$&$8k+7)$&$\frac{1}{2}k-2$\tabularnewline
$(3k-14,$&$3k+6,$&$6k-8)$&$(9,$&$6k-5,$&$6k+4)$&$(3k-7,$&$7k+3,$&$8k+5)$&$\frac{1}{2}k-1$\tabularnewline
$(3k-11,$&$3k+5,$&$6k-6)$&$(6,$&$6k-3,$&$6k+3)$&$(3k-4,$ & $7k+2,$ & $8k+3)$&$\frac{1}{2}k$\tabularnewline
$(3k-8,$&$3k+4,$&$6k-4)$&$(3,$&$6k-1,$&$6k+2)$&$(3k-1,$ & $7k+1,$ & $8k+1)$&$\frac{1}{2}k+1$\tabularnewline
$(3k-5,$&$3k+3,$&$6k-2)$&&&\multicolumn{1}{l}{}&&&&$\frac{1}{2}k+1$\tabularnewline
$(3k-2,$&$3k+2,$&$6k)$&&&\multicolumn{1}{l}{}&&&&$\frac{1}{2}k$\tabularnewline
\end{tabular}
\caption{Difference triples in columns and a coloring in rows.} \label{Table}
\endgroup
\end{table}

Figure \ref{Fig4} shows a chromatic class $X$ colored by the colors in box $1$. The remaining chromatic classes are obtained by rotation of $X$, the mapping $i\mapsto i+1$ mod $n$.

Finally, we get a proper coloring using $n(k/2+1)$ colors and the result follows.
\end{proof}

\begin{figure}
\begin{center}
\includegraphics[scale=.6]{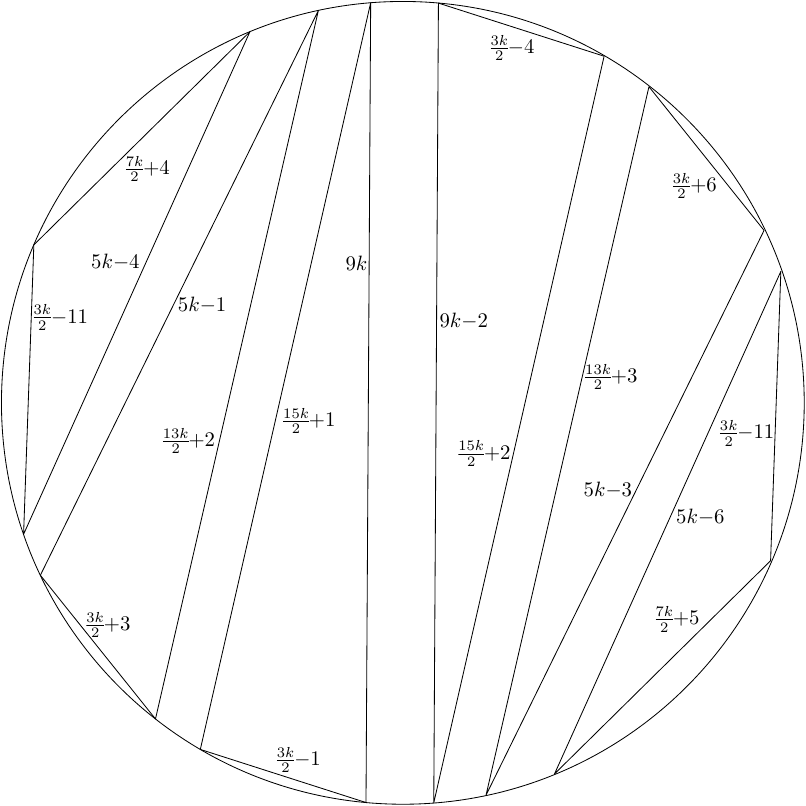}
\caption{A chromatic class of 6 triangles.}\label{Fig4}
\end{center}
\end{figure}

We now show that a quadratic number of colors is indeed needed for every decomposition of the complete convex geometric graph into triangles.

\begin{theorem}\label{theorem3.3}
 For every decomposition $[\mathsf{K}^\mathrm{c}_n,P]$ of $\mathsf{K}^\mathrm{c}_n$ such that each element of $P$ is a triangle, we have \[\chi'([\mathsf{K}^\mathrm{c}_n,P])\geq \frac{n^{2}}{119}-O(n).\]
\end{theorem}
 
\begin{proof}
Recall that the length of an edge $e_{i,j}$ is $\min\{|j-i|, n-|j-i|\}$. We define the \emph{length of a triangle} of $P$ as the length of its shortest edge.
Let $x$ be a real number to be determined later, and such that $x\geq3$. A triangle is called {\it{large}} if its length is at least $\frac{n}{x}$, and it is called {\it{short}} if its length is at most $\frac{n}{x}$.
 The number of edges with length at most $\frac{n}{x}$ is at most $\frac{n^2}{x}$, and therefore also the number of short triangles is at most $\frac{n^2}{x}$. Then, the number of large triangles is at least ${{n}\choose{2}}/3 - \frac{n^2}{x}$, where ${{n}\choose{2}}/3$ is the number of elements of $P$.
 
We now show that each chromatic class contains at most $x-2$ large triangles. Let $L$ be the set of large triangles of a chromatic class. Assume, to the contrary, that $|L|>x-2.$ Let $U$ be the subset of points of $S$ that are not vertices of triangles in $L$. Then, $n=|S|=|U|+3|L|$. Next we define a directed tree $T$, associated to $L$. The vertex set of $T$ consist of $|L|$ points placed in the interior of the triangles of $L$. Choose one of these points as the root vertex, denoted  $v$,  of $T$. We say that a point $a$ is {\it{visible}} from a point $b$ if the straight-line segment connecting $a$ and $b$ intersects no triangles from $L$ other than the two triangles which contain $a$ and $b$, respectively. Connect $v$ to all vertices visible from $v$, and connect each descendent of $v$ to its visible and not yet visited vertices. Iterate this process until all vertices of $L$ are visited. See Figure~\ref{Fig5}. Then, count the number $k$ of triangle edges of $L$ which are not intersected by edges of $T$ (drawn as bold edges in Figure~\ref{Fig5}). The triangles of $L$ have altogether $3|L|$ edges, $T$ has $|L|-1$ edges, and each edge of $T$ intersects two triangle edges. Therefore, $k \geq 3|L| - 2(|L|-1) = |L|+2$. Since we only consider large triangles, each of these $k$ edges, denoted $e'$, leaves at least $\frac{n}{x}-1$ points of $S$ on one side of $e'$ and leaves all triangles of $L$ on the other side of $e'$. Hence, for each edge $e'$ we count at least $\frac{n}{x}-1$ points of $U$. It follows that $|U|$ is at least $(|L|+2)(\frac{n}{x}-1)$. The assumption  $|L|>x-2$ then implies $|U| > |S|-(|L|+2)$. We get $|S|-3|L| = |U| > |S|-(|L|+2)$, which gives $|L|<1$, a contradiction.
 
We thus have the lower bound $$\chi'([\mathsf{K}^\mathrm{c}_n,P])\geq \frac{{{n}\choose{2}}/3 - \frac{n^2}{x}}{x-2}.$$
 Choosing $x=2(3+\sqrt{6})$ we obtain the claimed bound $$\chi'([\mathsf{K}^\mathrm{c}_n,P])\geq \frac{n^2}{60+24\sqrt{6}}-O(n) > \frac{n^2}{119}-O(n).$$ 
\begin{figure}
\begin{center}
\includegraphics[scale=0.8]{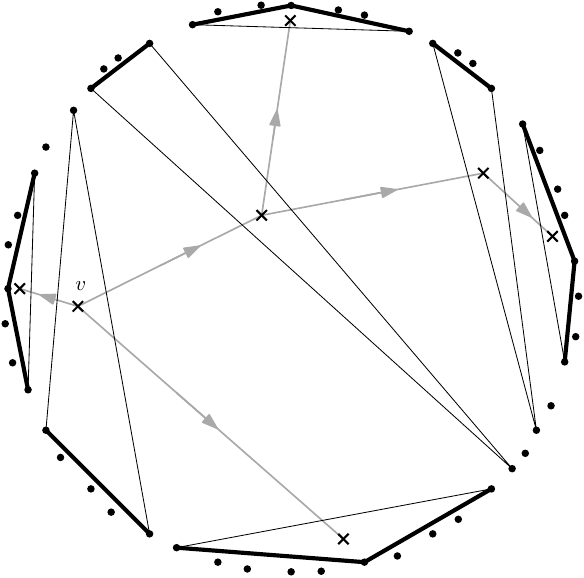}
\caption{The set $L$ of large triangles of a chromatic class and the associated directed tree $T$. Triangle edges not intersected by $T$ are drawn in bold.}\label{Fig5}
\end{center}
\end{figure}
\end{proof}


\section{The EFL conjecture for $\mathsf{K}_n$}\label{section4}

All the previous results support the following conjecture, that we called \emph{EFL conjecture for complete geometric graphs}:

\begin{conjecture}
Let $[\mathsf{K}_n,P]$ be a decomposition of $\mathsf{K}_n$, then \[\chi'([\mathsf{K}_n,P])\leq \frac{n^{2}}{9}+\Theta(n).\]
\end{conjecture}

Additionally, we want to note some remarks.

First, we mainly use a set of geometric complete subgraphs $W$ of $P$ which are pairwise intersecting  to give a lower bound for $\chi([\mathsf{K}_n,P])$. We denote as 
$\omega'([\mathsf{K}_n,P])$ the maximum number $k$ of elements of any $W$ and call it the \emph{clique index} of $[\mathsf{K}_n,P]$. Since $\omega'([\mathsf{K}_n,P])\leq\chi'([\mathsf{K}_n,P]$, a weaker conjecture is the following.

\begin{conjecture}
Let $[\mathsf{K}_n,P]$ be a decomposition of $\mathsf{K}_n$, then \[\omega'([\mathsf{K}_n,P])\leq \frac{n^{2}}{9}+\Theta(n).\]
\end{conjecture}

Second, consider a decomposition $[\mathsf{K}_n, P]$ of $\mathsf{K}_n$, and take the convex hull of each element in $P$, denote this new decomposition as  $[\mathsf{K}_n,\overline{P}]$. Under this new definition we say that two elements of $P$ intersect if their convex hulls intersect. Since $\chi'([\mathsf{K}_n,P]\leq \chi'([\mathsf{K}_n,\overline{P}])$, a stronger conjecture is the following.

\begin{conjecture}\label{conjecture3}
Let $[\mathsf{K}_n,\overline{P}]$ be a decomposition of $\mathsf{K}_n$, then \[\chi'([\mathsf{K}_n,\overline{P}])\leq \frac{n^{2}}{9}+\Theta(n).\]
\end{conjecture}

If we consider the clique index of $\chi'([\mathsf{K}_n,\overline{P}])$, we have that \[\omega'([\mathsf{K}_{n},P])\leq\begin{array}{c}
\chi'([\mathsf{K}_{n},P])\\
\omega'([\mathsf{K}_{n},\overline{P}])
\end{array}\leq\chi'([\mathsf{K}_{n},\overline{P}]).\]

Therefore, a non-comparative conjecture is the following.

\begin{conjecture}\label{conjecture4}
Let $[\mathsf{K}_n,P]$ be a decomposition of $\mathsf{K}_n$, then \[\omega'([\mathsf{K}_n,\overline{P}])\leq \frac{n^{2}}{9}+\Theta(n).\]
\end{conjecture}

Conjectures \ref{conjecture3} and \ref{conjecture4}, let $\tau(\mathsf{K}_n,p)$ be the largest number of edge-disjoint triangles in any partition $P$ containing some fixed point $p$ such that $p$ is not a vertex of $\mathsf{K}_n$. Hence $\tau(\mathsf{K}_n,p)\leq\omega'([\mathsf{K}_n,\overline{P}])$. In \cite{MR3205151} was proved that, for any point $p$, \[\tau(\mathsf{K}_n,p)\leq \frac{n^{2}}{9}+\Theta(n).\]

Also, it is possible to deduce the following equation from \cite{MR3205151} getting a similar result to Theorem \ref{theorem3}: For any natural number $n$ there exists a decomposition $[\mathsf{K}_n,P]$ of $\mathsf{K}_n$ and some point $p$ such that \[\frac{n^{2}}{12}-\Theta(n) \leq \tau(\mathsf{K}_n,p) \leq \omega'([\mathsf{K}_n,\overline{P}]) \leq \chi'([\mathsf{K}_n,\overline{P}]).\]

Several related coloring problems for triangles have been studied, see \cite{MR3030320}.
Finally, we propose the following problem.

\begin{problem}
Let $[\mathsf{K}_n,P]$ be a decomposition of $\mathsf{K}_n$. How many triangles must contain $P$, such that $P$ contains two disjoint triangles?
\end{problem}

If we change the question using ``edges'' instead of ``triangles'', the answer is $n+1$ and it was given by Erd{\H o}s  \cite{MR0015796}. We conjecture the following due to the fact that in Theorem \ref{theorem4} it is possible to give a decomposition $[\mathsf{K}^\mathrm{c}_n,P]$ of $\mathsf{K}^\mathrm{c}_n$ with $(\frac{n}{3})^{2}+\epsilon$ pairwise intersecting triangles, where $\epsilon=1$ if $3|n$.

\begin{conjecture}\label{conjecture5}
Let $[\mathsf{K}_n,P]$ be a decomposition of $\mathsf{K}_n$. If $P$ contains at least $(\frac{n}{3})^{2}+2$ triangles then $P$ contains at least two disjoint triangles.
\end{conjecture}

\section*{Acknowledgments}

We thank an anonymous referee for helpful suggestions.

C. Huemer was supported by projects MINECO MTM2015-63791-R and Gen. \ Cat. \ DGR 2017SGR1336. C.Rubio-Montiel was partially supported by a CONACyT-M{\'e}xico Postdoctoral fellowship, by the National scholarship programme of the Slovak republic and PAIDI/007/19. 


\bibliographystyle{amsplain}

\end{document}